\documentclass[11pt]{article}
\usepackage{amssymb,amsmath,amsthm,hyperref,verbatim,fullpage}

\newtheorem{theorem}{Theorem}

\newcommand{\F}{\mathcal{F}}
\newcommand{\G}{\mathcal{G}}

\newcommand{\Int}{\mathrm{Int}}
\newcommand{\ucstar}{\textrm{UC}^*}
\begin{document}
\title{A note on transitive union-closed families.}
\author{James Aaronson\footnote{Mathematical Institute, University of Oxford.},\, David Ellis\footnote{School of Mathematics, University of Bristol.}\, and Imre Leader\footnote{Department of Pure Mathematics and Mathematical Statistics, University of Cambridge.}}
\date{17th October 2020}
\maketitle

\begin{abstract}
We show that the Union-Closed Conjecture holds for the union-closed family generated by the cyclic translates of any
fixed set.
\end{abstract}

\section{Introduction}

If $X$ is a set, a family $\mathcal{F}$ of subsets of $X$ is said to be {\em union-closed} if the union of any two sets in $\F$ is also in $\F$. The celebrated Union-Closed Conjecture (a conjecture of Frankl \cite{duffus}) states that if $X$ is a finite set and $\F$ is a union-closed family of subsets of $X$ (with $\F \neq \{\emptyset\}$), then there exists an element $x \in X$ such that $x$ is contained in at least half of the sets in $\F$. Despite the efforts of many researchers over the last forty-five years, and a recent Polymath project \cite{polymath} aimed at resolving it, this conjecture remains wide open. It has only been proved under very strong constraints on the ground-set $X$ or the family $\F$; for example, Balla, Bollob\'as and Eccles \cite{bbe} proved it in the case where $|\F| \geq \tfrac{2}{3} 2^{|X|}$; more recently, Karpas \cite{karpas} proved it in the case where $|\F| \geq (\tfrac{1}{2}-c)2^{|X|}$ for a small absolute constant $c >0$; and it is also known to hold whenever $|X| \leq 12$ or $|\F| \leq 50$, from work of Vu\v{c}kovi\'c and \v{Z}ivkovi\'c \cite{vz} and of Roberts and Simpson \cite{rs}. We note that Reimer \cite{reimer} proved that the average size of a set in an arbitrary finite union-closed family $\F$ is at least $\tfrac{1}{2} \log_2(|\F|)$; this yields (by averaging) a good approximation to the Union-Closed Conjecture in the case where $\F$ is large, e.g.\ it implies that there is an element contained in at least an $\Omega(1)$-fraction of the sets in $\F$, in the case where $|\F| = 2^{\Omega(n)}$.

In this note, we prove the conjecture in the special case where $X$ is $\mathbb{Z}_n$, the cyclic group of order $n$, and $\mathcal{F}$ consists of all unions of cyclic translates of some fixed set. This is a question asked in the Polymath project \cite{polymath}.
\begin{theorem}
\label{thm:cyclic}
Let $n \in \mathbb{N}$, and let $R \subset \mathbb{Z}_n$ with $R \neq \emptyset$. Let $\mathcal{F}= \{A+R:\ A \subset \mathbb{Z}_n\}$ be the family of all unions of cyclic translates of $R$. Then the average size of a set in $\mathcal{F}$ is at least $n/2$. In particular, the Union-Closed Conjecture holds for $\mathcal{F}$.
\end{theorem}
Our proof is surprisingly short. In fact, we establish the following slightly more general result.
\begin{theorem}
\label{thm:main}
Let $(G,+)$ be a finite Abelian group, and let $R \subset G$ with $R \neq \emptyset$. Let $\mathcal{F} = \{A+R:\ A \subset G\}$ be the family of all unions of translates of $R$. Then the average size of a set in $\mathcal{F}$ is at least $|G|/2$. In particular, the Union-Closed Conjecture holds for $\mathcal{F}$. 
\end{theorem}

We remark that it is possible to deduce a slightly weaker form of Theorem \ref{thm:main} from a theorem of Johnson and Vaughan (Theorem 2.10 in \cite{jv}). In fact, the result of Johnson and Vaughan, after applying a quotienting argument,
yields that there is an element of $G$ contained in at least $(|\F|-1)/2$ of the sets in $\F$. (Since $\F$ may have odd size, for example when $G$ is $\mathbb{Z}_3$ and $R=\{0,1\}$, this is not quite enough to deduce Theorem \ref{thm:main}.) We are indebted to Zachary Chase for bringing this paper of Johnson and Vaughan to our attention.

A short explanation of our notation and terminology is in order. As usual, if $G$ is an Abelian group, and $A,B \subset G$, we write $A+B = \{a+b:\ a\in A,\ b \in B\}$ for the {\em sumset} of $A$ and $B$. Similarly, if $a \in G$ and $B \subset G$, we define $a+B = \{a+b:\ b \in B\}$. For any $x \in G$, we let $-x$ denote the inverse of $x$ in $G$, and for any set $A \subset G$, we let $-A = \{-a:\ a \in A\}$. We say a subset $A \subset G$ is {\em symmetric} if $A = -A$. If $X$ is a finite set, we write $\mathcal{P}(X)$ for the power-set of $X$. 

\section{Proof of Theorem \ref{thm:main}.}

Before proving Theorem \ref{thm:main}, we introduce some useful concepts and notation. Let $G$ be a fixed, finite Abelian group, and let $R \subset G$ be fixed. For any set $A \subset G$, we define its {\em $R$-neighbourhood} to be
$$N_{R}(A) := A+R,$$
and its {\em $R$-interior} to be
$$\Int_{R}(A): = \{x \in G:\ x+R \subset A\}.$$
We note that, if $R$ is symmetric and contains the identity element $0$ of $G$, then the $R$-neighbourhood of any set $A$ is precisely the graph-neighbourhood of $A$ in the Cayley graph of $G$ with generating-set $R \setminus \{0\}$, and similarly, the $R$-interior of $A$ is precisely the graph-interior of $A$ with respect to this Cayley graph.

\begin{proof}[Proof of Theorem \ref{thm:main}.]
Let $G$ be a fixed, finite Abelian group and let $R \subset G$ be a fixed, nonempty subset of $G$. Let
$$\mathcal{F} = \{A+R:\ A \subset G\}$$
be the union-closed family consisting of all unions of translates of $R$.

We define a function $f:\mathcal{P}(G) \to \mathcal{P}(G)$ by
$$f(S) = -(G \setminus \Int_R(S))\quad \text{for all }S \subset G.$$
It is clear that for any set $S \subset G$, $|\Int_{R}(S)| \leq |S|$, since for any element $r \in R$, the function $x \mapsto x+r$ is an injection from $\Int_R(S)$ into $S$. Hence,
\begin{equation}
\label{eq:sum}
|S|+|f(S)| \geq |G|\quad \text{for all } S \subset G.
\end{equation}
Next, we observe that
\begin{equation}
\label{eq:image}
f(S) = (-(G \setminus S))+R\quad \text{for all } S \subset G.\end{equation}
Indeed, for any $x \in G$, it holds that $x \in f(S)$ iff $-x \notin \Int_R(S)$ iff $(-x +R) \cap (G \setminus S) \neq \emptyset$ iff $x \in (-(G \setminus S))+R$. It follows that $f(\mathcal{P}(G))\subset \mathcal{F}$.

Finally, we observe that the restriction $f|_{\mathcal{F}}$ is an injection. This might seem surprising at first glance, but it follows immediately from the fact that
\begin{equation}
\label{eq:eq}
N_{R}(\Int_{R}(A+R)) = A+R \quad \text{for all } A \subset G.
\end{equation}
To see (\ref{eq:eq}), let $S = A+R$ and observe that $N_{R}(\Int_{R}(S)) \subset S$ holds by definition (in fact for any set $S$). On the other hand, if $S = A+R$, then we have $A \subset \Int_{R}(S)$ and therefore $S=A+R \subset N_{R}(\Int_R(S))$. Hence, $S = N_{R}(\Int_{R}(S))$, as required.

Putting everything together, we see that $f|_{\mathcal{F}}$ is a bijection from $\mathcal{F}$ to itself and satisfies
$$|S|+|f(S)| \geq |G|\quad \text{for all }S \in \mathcal{F}.$$
Therefore,
$$\frac{1}{|\mathcal{F}|}\sum_{S \in \mathcal{F}} |S| = \frac{1}{2|\mathcal{F}|} \sum_{S \in \F}(|S|+|f(S)|) \geq \frac{1}{2|\mathcal{F}|}\sum_{S \in \F} |G| = |G|/2,$$
proving the first part of the theorem. It follows that
$$\frac{1}{|G|} \sum_{x \in G} \frac{|\{S \in \F:\ x \in \F\}|}{|\F|} = \frac{1}{|G|} \frac{1}{|\F|} \sum_{S \in \F}|S| \geq 1/2,$$
so by averaging, there exists $x \in G$ such that at least half the sets in $\F$ contain $x$, and so the Union-Closed Conjecture holds for $\mathcal{F}$.
\end{proof}

\end{document}